\documentclass[12pt,oneside]{amsart}

\usepackage{rotating}

\usepackage{amsmath}
\usepackage{amssymb, latexsym, amsfonts, amscd, amsthm, mathrsfs, enumerate, esint}
\usepackage[usenames,dvipsnames]{color}
\usepackage[shortlabels]{enumitem}
\usepackage{graphicx}
\usepackage{mathtools} 
\usepackage[colorlinks=true,allcolors=blue]{hyperref}
\usepackage{url}
\usepackage{stmaryrd}
\usepackage[normalem]{ulem}

\usepackage{cite}

\usepackage{geometry}
\usepackage{chngcntr}
\counterwithin{equation}{section}
\usepackage{setspace}
\newcommand*{\doi}[1]{\href{http://dx.doi.org/\detokenize{#1}}{doi}}
\allowdisplaybreaks
\usepackage{amsmath}
\usepackage{amssymb,amsbsy,amsthm}
\usepackage{graphicx}

\usepackage{enumerate}
\usepackage{cases}

\usepackage{hyperref}

\allowdisplaybreaks[1]

\DeclareMathOperator{\N}{\mathbb{N}} 
\DeclareMathOperator{\Z}{\mathbb{Z}} 

\newcommand{\p}{\mathbb{P}}
\newcommand{\J}{k}
\renewcommand{\P}{\mathbb{P}}
\newcommand{\E}{\mathbb{E}}
\newcommand{\tE}{\widetilde{\E}}
\newcommand{\tP}{\widetilde{\mathbb{P}}}

\newcommand{\dn}{D_n}
\newcommand{\di}{D_i}
\newcommand{\w}{\mathrm{Hole}}
\newcommand{\whole}{W}
\newcommand{\tstar}{\mathfrak{t}}
\newcommand{\sleep}{\mathfrak{s}}
\newcommand{\mn}{\mathrm{Exit}_{\dn}}
\newcommand{\sn}{\mathrm{Frozen}_{\dn}}

\renewcommand{\leq}{\leqslant}
\renewcommand{\geq}{\geqslant}

\newcommand{\cF}{\mathcal{F}}
\newcommand{\cG}{\mathcal{G}}

\newtheorem{theorem}[equation]{Theorem}
\newtheorem{lemma}[equation]{Lemma}

\newtheorem{proposition}[equation]{Proposition}

\mathtoolsset{showonlyrefs}

\begin{document}

\title
{Active Phase for Activated Random Walk on $\mathbb{Z}$}
\author{Christopher Hoffman \and Jacob Richey \and Leonardo T. Rolla}
\date{\today}



\begin{abstract}
We consider the Activated Random Walk model on $\mathbb{Z}$.
In this model, each particle performs a continuous-time simple symmetric random walk, and falls asleep at rate $\lambda$.
A sleeping particle does not move but it is reactivated in the presence of another particle.
We show that for any sleep rate $\lambda < \infty$  if the density $ \zeta $ is close enough to $1$ then the system stays active.
\end{abstract} 

\maketitle

\section{Introduction}

Activated Random Walk (ARW) is a reaction-diffusion type of particle system that has been intensively studied in recent years. The ARW model belongs to a
wider class of Abelian networks that are believed to exhibit self organized criticality. Finite versions of these Abelian networks are Markov chains that have a density parameter that can vary over time. The prediction of self organized criticality says that the density in the Markov chain is naturally attracted to the critical density. 

For many Abelian networks there are multiple ways to define a critical density. (For example, by using finite or infinite systems and by varying the starting configuration.)  
For ARW all of these different definitions appear to produce the same critical value.
Such a system is said to exhibit \emph{universality}.
Within the realm of Abelian networks, ARW has generated particular interest as it believed to both exhibit universality and to have a non-trivial critical state.
There has been some recent progress towards establishing universality for ARW~\cite{RollaSidoraviciusZindy19}.

So far much of the work on ARW has been dedicated to proving non-trivial bounds on the critical density~\cite{RollaSidoravicius12,BasuGangulyHoffman18,StaufferTaggi18}.
ARW on a graph can be parameterized by two parameters, a sleep rate $ \lambda $ and initial density $ \zeta $.
For every sleep rate $ \lambda $ and initial density $ \zeta $ the system either fixates or stays active.  For a fixed sleep rate $ \lambda $ remaining active is monotone in the density so there exists a critical density, below which the system fixates and above which the system stays active. See Figure~\ref{fig:murph}. More  precise definitions are given in Section~\ref{sec:setup}. It is highly nontrivial to show that the critical density is neither 0 nor 1. 

The most difficult cases for analyzing ARW have proven to be those where the underlying walk is recurrent, namely unbiased walks in dimensions $ 1 $ and $ 2 $. In dimension 1 the third author and Sidoravicius established that for any density the critical value is greater than zero~\cite{RollaSidoravicius12}. 
Basu, Ganguly and the third author proved that if the sleep rate is low then critical value is less than one~\cite{BasuGangulyHoffman18}. 
In this paper we establish the missing part of the active phase for the one-dimensional case: we show existence of a non-trivial active phase for arbitrarily large values of the sleep rate $ \lambda $. Thus the curve in Figure~\ref{fig:murph} lies strictly between the vertical lines $\zeta=0$ and $\zeta=1$.
The two-dimensional case remains widely open.

\begin{figure}
\includegraphics[page=2,width=.95\textwidth]{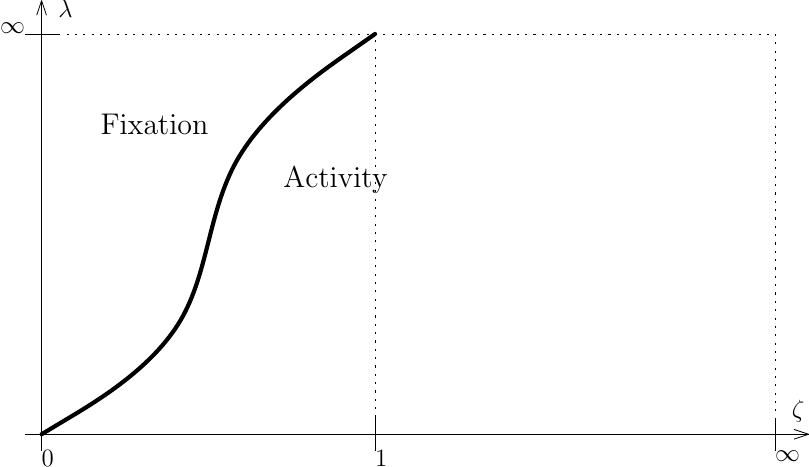}
\par\vspace{1em}
\caption{The phase diagram for ARW on $\Z$. In this paper we show that for every value of the sleep rate $\lambda$ the critical density $\zeta$ is less than one.}
\label{fig:murph}
\end{figure}

Our arguments follow the breakthrough work of~\cite{BasuGangulyHoffman18} which introduced a block argument to obtain some decoupling and make the dynamics amenable to analysis. We build upon a reformulation of the block approach presented in~\cite{Rolla19} and utilized in~\cite{AsselahSchapiraRolla19}.
However, there are serious challenges in analyzing the block dynamics without assuming $ \lambda $ to be small.

A brief description of the dynamics is the following.
At $ t=0 $, each site $ x \in \Z $ has a random number of particles, with an average $ \zeta>0 $ and (for simplicity) independent of the other sites.
Active particles perform continuous-time simple symmetric random walks, independently of each other.
At rate $ \lambda>0 $, an active particle decays to a sleeping state.
A sleeping particle does not jump, and remains in that state indefinitely, until the moment when another particle is present at the same site.
See~\cite{Rolla19} for more details and a construction of this interacting particle system, as well as motivation and physical background.

We say that the system \emph{stays active} if each site is visited by active particles infinitely many times.

\begin{theorem}
\label{main_thm}
For every $\lambda < \infty$, there exists $\zeta < 1$ such that the one-dimensional simple symmetric ARW with sleep rate $ \lambda $ and initial density $ \zeta $ will a.s.\ stay active.
\end{theorem}
This result extends \cite{BasuGangulyHoffman18} which proved a the same result for small values of $\lambda$.

The study of ARW has seen considerable progress in the last decade~\cite{AmirGurel-Gurevich10,AsselahSchapiraRolla19,BasuGangulyHoffman18,BasuGangulyHoffmanRichey19,CabezasRolla20,CabezasRollaSidoravicius14,CabezasRollaSidoravicius18,DickmanRollaSidoravicius10,RollaSidoravicius12,RollaSidoraviciusZindy19,RollaTournier18,Shellef10,SidoraviciusTeixeira17,StaufferTaggi18,Taggi16,Taggi19,Taggi20}. Most fruitful approaches so far rely on the Abelian property of ARW's site-wise representation to design ``toppling procedures'' that provide partial information about the system. See~\cite{Rolla19} for a recent and comprehensive  account. Also see~\cite{BondLevine16,CandelleroGangulyHoffmanLevine17,Jarai18,CandelleroStaufferTaggi20} for recent results about other Abelian networks.


Conservation of particles in this model leads to long-range space-time correlations that make the ARW a difficult model to study by standard methods.
The argument introduced in~\cite{BasuGangulyHoffman18} gets around this problem by dynamically attaching particles to a sparse set of sites, $K\Z$ for some large value of $K$, called sources.
It also chooses $\lambda$, depending on $ K $, so  small that a particle can move from source to source with a small probability of falling asleep. They analyze the odometer function (the count of how many times a particle has moved from each source to each of its neighbors) using an energy-entropy calculation.


This approach from \cite{BasuGangulyHoffman18} was reformulated in~\cite{Rolla19} as having two separate parts: a mass balance equation and the probabilistic analysis of a single block.
Using this reformulation, a quantitative improvement of the estimates from~\cite{BasuGangulyHoffman18} were obtained in~\cite{AsselahSchapiraRolla19}, finding the right order of $ \lambda $ for small $ \zeta $ up to a constant factor.

Contrary to the regime considered here, this program worked well for very small values of $ \lambda $, where a particle released at a certain site has high probability of traveling a long distance before it falls asleep.
%
To translate those ideas to the case of large $ \lambda $
a natural starting point is to think of $ \zeta \approx 1 $ as a configuration with most sites occupied by a single particle, which provide a safe ``carpet'' for other particles to walk over, and few vacant sites where these wandering particles will be challenged by the high sleep rate $ \lambda $.
Sections~\ref{sec:carpet_hole} and~\ref{sec:mbe} are devoted to develop a scheme that  allows us to introduce dynamics which allow for similar calculations as those in \cite{BasuGangulyHoffman18}.
%

This paper is divided as follows. In Section~\ref{sec:setup} we briefly describe the site-wise representation and Abelian property. Section~\ref{sec:carpet_hole} is devoted to describing our toppling procedure in detail, including the stationary `carpet' initial condition. In Section~\ref{sec:mbe} we introduce the relevant sigma-algebras, particle counters, and mass balance equations, use them to state the single block estimate, and use the latter to prove Theorem~\ref{main_thm}. Finally, Section~\ref{omari} is devoted to proving the single block estimate.

\section{Formal definitions and setup} \label{sec:setup}


The ARW dynamics is defined as follows. 
The initial configuration is sampled from an ergodic measure $\nu$ on $\Z$ with mean $\zeta$ particles per site.
Particles can be active ($A$) or sleeping ($S$); initially all particles are active.
Each particle performs simple symmetric random walk at rate $ 1 $, and at each site, the transitions $A + S \to 2A$ and $A \to S$ occur at rates $\infty$ and $\lambda$, respectively.

More formally, the configuration of the system at each time $ t \geq 0 $ is given by $ \eta_t \in \{0, \sleep, 1, 2, \ldots\}^\mathbb{Z} $, where $\sleep$ is a formal symbol representing a single sleeping particle.
For each $ x \in \mathbb{Z} $, the system undergoes the transitions $ \eta_t = \tstar_{x, x-1} \eta_{t-} $ and $ \eta_t = \tstar_{x, x+1} \eta_{t-} $ at rates equal $ \frac{1}{2} $ times the number of active particles at $ x $, and transition $ \eta_{t} = \tstar_{x,\sleep} \eta_{t-} $ at rate $ \lambda $ times the same number.
The transitions $ \tstar $ are defined by
\begin{equation}
\tstar_{x,\sleep}(\eta)(z) =
\begin{cases}
\sleep, & z = x \text{ and } \eta(x) = 1, \\
\eta(z), & \text{otherwise},
\end{cases}
\end{equation}
and
\begin{equation}
\tstar_{x, y}(\eta)(z) =
\begin{cases}
\eta(y) + 1, & z = y, \\
\eta(x) - 1, & z = x, \\
\eta(z), & \text{ otherwise,} 
\end{cases}
\end{equation}
where we define $ \sleep + 1 = 2 $ to provide the $ A+S \to 2A $ reaction.

A useful description is the site-wise representation.
Each site $x$ has associated to it an infinite \emph{stack} of instructions $(\xi_{(x,k)})_{k = 1, 2, \ldots}$, which are sampled i.i.d.\ over $ x $ and $ k $, so that
\begin{equation}
\xi_{(x,k)} =
\begin{cases}
\tstar_{x, x+1} & \text{ with probability } \frac{1}{2(1+\lambda)}, \\
\tstar_{x, x-1} & \text{ with probability } \frac{1}{2(1+\lambda)}, \\
\tstar_{x, \sleep} & \text{ with probability } \frac{\lambda}{1+\lambda}.
\end{cases}
\end{equation}
We say that a site is stable if $ \eta(x) = 0 $ or $ \sleep $, and unstable otherwise.
The operation of \emph{toppling} a site $ x $ is \emph{legal} if $ x $ is unstable, and consists in updating the configuration by
$ \eta \to \xi_{(x,k)} \eta $, where $ k = \min \{j: \text{instruction } \xi_{(x,j)} \text{ has not been used}\} $.


It is not obvious that a process $ (\eta_0)_{t \geq 0} $ following the rules described above exists. But it turns out that it does exist and can be constructed explicitly using the stacks $ \xi $ and some Poisson clocks to trigger toppling operations~\cite{Rolla19}.
Let $ \p $ denote the underlying probability measure.

The main reason to use the site-wise representation is the Abelian property.
It says that any two sequences of instructions that stabilize the system on a given finite region yield the same final configuration.
We are not going to use the Abelian property directly, but rather a condition for non-fixation that is obtained using it.

Given $ k\in\N $ and a sequence $ \alpha = (x_1,x_2,\dots,x_k) $, we say that $ \alpha $ is a \emph{legal sequence of topplings} for $ \eta $ if $ x_1 $ is unstable, $ x_2 $ is unstable after $ x_1 $ is toppled, $ x_3 $ is unstable after $ x_1 $ and $ x_2 $ are toppled, and so on.
We define the odometer $ m_\alpha(x) $ as the number of times that $ x $ appears in $ \alpha $.
Finally, we define
\[
m_\eta(x) = \sup_\alpha m_\alpha(x)
\]
where the supremum is over all finite legal sequences of topplings.

\begin{lemma}
\label{lemma:criterion}
If the initial state $ \nu $ is spatially ergodic, then
\[
\p( (\eta_t)_{t \geq 0} \text{ stays active} )
=
\lim_k
\p( m_\eta(0) \geq k )
=
0 \text{ or } 1
.
\]
\end{lemma}



\begin{proof}
This is just a reformulation of~\cite[Corollary~2.8]{Rolla19}.
\end{proof}

Finally, we can take as our initial distribution any shift-invariant ergodic measure, thanks to the following result from~\cite{RollaSidoraviciusZindy19}.

\begin{theorem}
\label{any_ergodic}
For any $\lambda$, there exists $\zeta_c$ such that, for any shift-invariant ergodic distribution $\nu$ supported on active configurations with density $\zeta$, the Activated Random Walk model with initial condition sampled from $\nu$ and sleep rate $ \lambda $ will a.s.\ fixate if $\zeta < \zeta_c$ and a.s.\ stay active if $\zeta > \zeta_c$.
\end{theorem}

This allows us to choose a particularly useful initial condition, which is described in Section~\ref{sec:carpet_hole}.

\section{Carpet-hole toppling procedure}
\label{sec:carpet_hole}

We perform a particular sequence of topplings where particles and sites are grouped into large blocks.
We work at one block at a time, following a left-to-right policy.

This `carpet-hole toppling procedure' classifies particles into two types -- `free' and `carpet.' The free particles will be further subdivided into `hot,' `thawed' and `frozen'. 
There are also two types of regions of space -- `blocks' and `transit regions' -- that are fixed throughout the process, and a few special sites called `holes,' which constantly move inside each block.

\subsection{Initial configuration on $\Z$}

For a given $\lambda < \infty$, we will take
$a = a(\lambda)$ and $K = K(a,\lambda) > a$ large (to be chosen later in \eqref{aclu}).
Consider the following configuration of particles
\begin{equation}
\eta_0^{\rm neat}(x) =
\begin{cases}
0, & x \in 2K\Z, \\
1, & \text{otherwsise},
\end{cases}
\end{equation}
shown in Figure~\ref{fig:configurations1}.

So initially every site has exactly one particle, except for the sites $iK$ for $i = 2, 4, 6,\dots$, which have zero particles.
The initial configuration is a random translation of $ \eta_0^{\rm neat} $:
\begin{equation} 
\nu_0 = \frac{1}{2K} \sum_{j=1}^{2K} \delta_{\theta^j \eta_0^{\rm neat}},
\end{equation}
where $\theta$ is the usual shift map. Since $\nu_0$ is shift-invariant ergodic, using Theorem~\ref{any_ergodic} allows us to start from this initial configuration in order to prove the main result. 
Observe that the density of particles in $\eta_0^{\rm neat}$ is $\zeta = 1-(2K)^{-1}$.

\subsection{Toppling procedure on finite configurations}
\label{sub:procedure}

We now specify a procedure to decide which unstable sites to topple, so as to obtain a lower bound in probability to the number of particles that exit a large interval.

Fix $ n \in 2\N $, and let $$ \dn = (a,nK+K-a) .$$
For each $i = 1, 2, \ldots, n$, the segment $[iK-a, iK+a]$ is called the $i$th \emph{block} and $(iK+a, iK+K-a)$ is called the $i$th \emph{transit region}. The sites in $K\mathbb{Z}$ are called \emph{holes}. These holes may move around during the procedure but they always stay within their block.

We divide the particles into several categories and particles will be able to switch back and forth between the categories by rules which we will now describe. Each particle will either be a \emph{free} particle or a \emph{carpet} particle.
Initially, particles at sites $K, 3K, 5K, \ldots$ are called \emph{free} particles, and all other particles are called \emph{carpet} particles. All free particles start \emph{thawed}, and the one at site $ x=K $ is declared \emph{hot}.


\begin{sidewaysfigure}
\vspace*{.65\textwidth}
\includegraphics[page=1,width=\textheight]{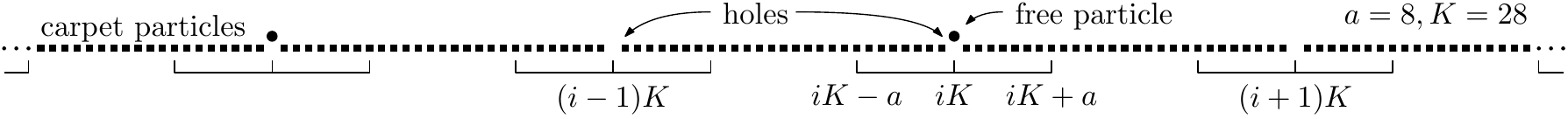}
\par\vspace{1em}
\caption{A portion of the initial configuration which contains four blocks separated by three transit regions. All three blocks have holes which are located at $(i-2)K$, $(i-1)K$, $iK$ and $(i+1)K$. The two holes at $(i-1)K$ and $(i+1)K$ have no particles.}
\label{fig:configurations1}
\par\vspace{7em}
\includegraphics[page=2,width=\textheight]{figures/configurations}
\par\vspace{1em}
\caption{Example of configuration during the procedure. This choice of hot particle assumes that all free particles to the left of the displayed region are frozen.}
\label{fig:configurations2}
\end{sidewaysfigure}

The following properties are trivially satisfied by $ \eta_0^{\rm neat} $ (as well as other configurations $ \eta^{\rm neat}(m, i) $ that we will introduce later), and will be preserved by the procedure (see Figure~\ref{fig:configurations2}):
\begin{enumerate}[(P1)]
\item
\label{prop:first}
\label{prop:onehole}
Each block~$ i $ has exactly one hole which is located at some site $ x \in [iK,iK+a] $.
\item
\label{prop:hascarpet}
Every site except the holes contains a carpet particle.
\item
\label{prop:carpetactive}
Carpet particles between the hole and $ x=iK+a $ are active.
\item
Free particles are always active.
\item
\label{prop:freewhere}
All free particles except the hot particle are at sites $ iK $ or $ iK \pm a $, and there is at most one free particle at $ iK $ for each $ i $.


\item
\label{prop:atmostone}
There is at most one frozen free particle per block.
\item
\label{prop:frozenatikpa}
There is a frozen free particle in block~$ i $ if and only if the hole and the frozen particle are both at position $ iK+a $.
\item
\label{prop:holeempty}
In the block containing the hot particle, the hole (which by~\ref{prop:onehole} is unique) is either vacant or occupied by the hot particle itself, unless the hole is at $ x=iK+a $.
\item
\label{prop:hotisfree}
The hot particle is free and thawed.
\label{prop:last}
\end{enumerate}

We only move the hot particle.
We will follow a \emph{left-most priority policy} for choosing which free particle in $ \dn $ will be hot, described as follows.
To attempt an emission, we choose the left-most block (smallest $ i $) that contains a thawed free particle.
Among the thawed free particles in block~$ i $, we choose the one at $ iK $ if there is one, otherwise one from $ iK-a $. If neither of these two sites has a free particle, by~\ref{prop:freewhere} we choose one from $ iK+a $. We declare the chosen particle to be \emph{the hot particle}.
This seemingly strange choice rule inside the block is to ensure~\ref{prop:holeempty}.
After each successful or failed attempted emission from a block, we choose the next hot particle (possibly the same!) with the same criterion.

\subsection{Attempted Emissions.}
Suppose we have just declared a particle to be hot, say at block~$ i $.
We now outline a portion of the carpet-hole toppling procedure that we call an \emph{attempted emission}. This is the evolution of the carpet-hole toppling procedure until the hot particle reaches a new block or until the hot particle is frozen. We call the arrival of the hot particle at either block~$i-1$ or block~$i+1$ a \emph{successful emission}.
We call an attempted emission ending with the hot particle being frozen a \emph{failure}. Note that each attempted emission is associated with one block so it makes sense to speak of the attempted emissions in block~$i$.

{\bf Case 1: There is a frozen particle at block~$ i $. }
In this case, by ~\ref{prop:onehole},~\ref{prop:hascarpet},~\ref{prop:frozenatikpa}, and~\ref{prop:hotisfree}, every site in block~$ i $ has a particle which is not the hot particle. We repeatedly topple the site that contains the hot particle until it arrives at $ iK \pm K \mp a $. This is a \emph{successful emission}.

Moreover, if in this process the hot particle visits every site in $ [iK,iK+a] $, we move the hole to position $ iK $, turn the frozen free particle at $iK+a$ into a carpet particle, and turn the carpet particle at site $ x=iK $ into a thawed free particle.
Properties~\ref{prop:first}-\ref{prop:last} are preserved (using~\ref{prop:atmostone}, there are no other frozen particles in block~$ i $ besides the one we have just thawed).

{\bf Case 2: There is not a frozen particle at block~$ i $. }
This case is more common and more involved to describe.
By~\ref{prop:frozenatikpa} and~\ref{prop:holeempty}, the hole is vacant.
We repeatedly topple the hot particle until it either enters a neighboring block at $ iK \pm K \mp a $ (which finishes the emission attempt with success) or arrives at the hole.

Once at the hole, we topple the hot particle.
If it sleeps, we turn it into a carpet particle, turn the next site to the right into the hole and its carpet particle into the hot particle (by~\ref{prop:carpetactive}, it is active).
If this happens at site $ iK+a-1 $, so that the hole moves to $ iK+a $, the emission attempt has failed, the hot particle is declared frozen and (if possible) a new hot particle is to be chosen (so~\ref{prop:hotisfree} is not violated).

If it jumps to the right, it may either arrive at $ iK+K-a $, completing a successful emission, or return to the hole (back to the same status quo).
If it jumps to the left, it may either arrive at $ iK-K+a $, completing a successful emission, or return to the hole, in which case it will have visited a number if sites in $ [iK,iK+a] $.
In this case, we move the hole to leftmost site in $ [iK,iK+a] $ just visited, turn the hot particle into a carpet particle, and turn the carpet particle at the (new) position of the hole into the new hot particle, preserving~\ref{prop:carpetactive}.
This will be repeated until the hot particle is successfully emitted to a neighboring block or remains frozen at $ iK+a $.
Again, Properties~\ref{prop:first}-\ref{prop:last} are preserved.

{\bf Case 3: There are no thawed particles in $ \dn $. }
When there are no more free thawed particles in $ \dn $, then we cannot choose a new hot particle and the carpet-hole toppling procedure terminates.

\subsection{Conservation of free particles}

By following the above procedure, the number of free particles is conserved, as is the number of carpet particles. The number of holes is also conserved, and the hole in block~$i$ never leaves block~$i$. 
Moreover, every site between the hole and the right hand boundary $ iK+a $ contains an active carpet particle; thawed free particles are always active; and frozen free particles are always at the boundary of a block.

Let $\sn$ denote the number of frozen free particles remaining in $\dn$ after applying this toppling procedure to the first $n$ blocks. 
Also let $\mn$ denote the number of free particles that exit $\dn$ while applying this toppling procedure to the first $n$ blocks. 

\begin{lemma} \label{got nuts} The number of free particles never changes during the carpet-hole toppling procedure. It is always $n/2$ and
\begin{equation}\label{almond}\mn + \sn = \frac{n}{2}.\end{equation}
\end{lemma}

\begin{proof}
There are initially $n/2$ free particles in $\dn$. Whenever we change a free particle to a carpet particle then we declare a carpet to be a free particle. This is the only time that we change carpet particles to free particles. Thus the number of free particles never changes. After running the carpet-hole toppling procedure each free particle has either exited $\dn$ and contributed to $\mn$, or remains frozen inside $\dn$ and contributes to $\sn$.
\end{proof}

\subsection{Proof of Non-Fixation}

The main work in proving Theorem~\ref{main_thm} is to establish the following proposition.
\begin{proposition}
\label{prop:fewsleep}
Given $ \lambda<\infty $, there exist $c > 0, a \in \N$ and $K \in \N$ such that for $n$ sufficiently large, 
\begin{equation}
\p(\sn \geq n/4) \leq e^{-cn}.
\end{equation}
\end{proposition}

This quickly implies the main theorem, via Lemma~\ref{lemma:criterion}.

\begin{proof}
[Proof of Theorem~\ref{main_thm}]
Assume that $ \eta_0 = \eta_0^{\rm neat} $, which occurs with probability $ \frac{1}{2K} $.
First note that $ m_\eta(a) + m_\eta(b) \geq \mn $, where $ b=nK+K-a $.
This is because every toppling performed in this procedure is legal for the ARW.
When the procedure ends, there are at least $ \mn $ particles at $ \{a,b\} $ and these two sites can now be toppled at least as many times as the number of particles standing there.

Recall that there are initially $n/2$ free (non-carpet) particles in the carpet configuration. Lemma~\ref{got nuts} and Proposition~\ref{prop:fewsleep} yield
\begin{equation}
\p\left(\mn \geq n/4\right) = \p\left(\sn \leq n/4\right) > 1-e^{cn}
\end{equation}
for large $n$. 
Since $K$ is fixed, letting $n \to \infty$ and using Lemma~\ref{lemma:criterion} we conclude that the system will a.s.\ stay active.
\end{proof}

\section{Coupling, filtration and mass balance equations}
\label{sec:mbe}

In this section we give a proof of Proposition~\ref{prop:fewsleep}.
Our strategy is as follows.
$\sn$ is the number of blocks that have a frozen particle at the end of the carpet-hole toppling procedure, so it can be written as the sum of indicator functions of there being a frozen particle in each block. One possible approach to proving Proposition~\ref{prop:fewsleep} is to show that these events have small probability and are roughly independent. Unfortunately these events appear to be dependent in some extremely complicated way.
We get around this by writing $\sn$ as a sum of random variables in a different manner. To do this we will need to make some delicate choices of the instructions in the site-wise representation.
The choices of the instructions (and some associated $\sigma$-algebras that the choices generate) are defined in Sections~\ref{sami} and~\ref{tamini}.
Using these choices we show how to write $\sn$ as a sum in Lemma~\ref{eq:totalunavailable} of Section~\ref{coarse}.
Then in Section~\ref{chemtrails} we state Lemma~\ref{lem:singleblock} which shows that the random variables in the sum are only weakly dependent. This allows us to prove Proposition~\ref{prop:fewsleep}.
We delay the proof of Lemma~\ref{lem:singleblock} until Section~\ref{omari}.

\subsection{Coupling and filtration} \label{sami}

We start by describing a coupling which will introduce some statistical independence between the procedure on different blocks and allow us to disentangle their interaction to some extent.

For each $y \in \Z$ in any block, let $\xi^y = (\xi^y_k)_{k \in \N}$ be an i.i.d.\ sequence of instructions. When we topple the hot particle at site $y$, the particle moves according to the next unused instruction in this stack. 
For sites $y \in \Z$ in the transit regions, set $\xi^y = (\xi^{y, L}, \xi^{y, R})$, where each coordinate is an independent copy of the same stack of instructions. When the hot particle reaches a site $y$ in a transit region, we check which block that particle visited last. If it was the block to the left of $y$, then the particle moves according to next unused instruction in the $L$ stack, and otherwise it uses the next unused instruction in the $R$ stack. 


Consider the following filtration associated to the carpet-hole toppling procedure, which encapsulates the left-to-right procedure: 
\begin{equation}
\cF_i = \sigma[\{\xi^y: y \in (0, iK + a]\} \cup \{\xi^{x,L}: x \in (iK+a, iK+K-a)\}]
\end{equation}
In words, $\cF_i$ consists of all the information of all stacks in blocks up to block~$i$, plus all the stacks associated to particles exiting block~$i$ at its right endpoint until they reach block~$i+1$. 
So $ \cF_i $ tells us about particles that go from block~$i$ to $i+1$ but not vice versa.


\subsection{Relevant observables in each block} \label{tamini}

For any $m \in \N \cup \{0\}$, consider the configuration 
\begin{equation}
\label{eq:adding}
\eta_0^{\rm neat}(m, i) = \eta_0^{\rm neat} + m \delta_{iK+a}
\end{equation}
obtained by adding $m$ particles to the right boundary of block~$i$. Define 
$L_i^n(m)$ as the total number of left-step instructions executed from the stack $ \{\xi^{(i-1)K + a + 1, R}_k\}_k $
after performing the carpet-hole toppling procedure on blocks~$ 1, 2, \ldots, n $ starting from initial configuration
$ \eta_0^{\rm neat}(m, n) $.
So
$L_i^n(m)$ is
the number of times some particle is emitted from the $i$th block and arrives at the right boundary of the $(i-1)$st block if 
we apply the carpet-hole toppling procedure to the initial configuration
$ \eta_0^{\rm neat}(m, n) $.

Also, let $S_i^n(m)$ denote number of free particles in block~$ i $ and in state $ S \text{ after stabilizing blocks } 1, 2, \ldots, i $ starting from initial configuration $ \eta_0^{\rm neat}(m, n) $ according to the carpet-hole toppling procedure on $ \dn $.

When we consider $ n=i $, we may write $ S_i $ for $ S_i^i $ and $ L_i $ for $ L_i^i $.


\subsection{Coarse-grained particle flow} \label{coarse}

This description of individual blocks is completed by considering how they interact.
We consider the flow of particles between blocks, or more precisely the right-to-left flow.
Define the random ``coarse-grained odometer'' vector $\boldsymbol{M}^n = (M^n_0, M^n_1, \ldots, M^n_n)$, by letting $ M^n_i $ denote the
number of times a free particle is emitted from block
$ i+1 $
to block
$ i $
after running this procedure on $ \dn $. Although we do not use the following fact it is worthwhile to notice that this definition implies that $ M^n_n = 0$.

Fix any sequence of instructions.
When we apply the carpet-hole toppling procedure on $ \dn $ starting from configuration $ \eta_0^{\rm neat} $ (equivalently $ \eta_0^{\rm neat}(0,n) $) we get the following vectors:
\begin{itemize}
\item the coarse-grained odometer vector $(M^n_0, M^n_1, \ldots, M^n_n)$, 
\item the frozen particle vector $(S_0^n(0), S_1^n(0), \ldots, S_n^n(0))$ and 
\item the left emissions vector
$(L_0^n(0), L_1^n(0), \ldots, L_n^n(0))$.
\end{itemize}
For each $i=1,\dots,n$ fixed, we can also run the carpet-hole procedure on $ \di $ starting from configuration $ \eta_0^{\rm neat} (M^n_i,i)$ using the same stacks of instructions. This gives us the sequences $ S_i(M^n_i) $ and $ L_i(M^n_i) $.
The next lemma relates these vectors.


\begin{lemma}
\label{lemma:addsame}
For each~$ i \in \{1, \dots ,n\}$, a.s.,
 \begin{equation}\text{ $ S_i^n(0)=S_i(M^n_i) $ \ \ \ \ and \ \ \ \ \
$ L_i^n(0)=L_i(M^n_i) $.}
\end{equation}
\end{lemma}


\begin{proof}
The lemma is tautological for $ i=n $.
Fix $ i \in \{1,\dots,n-1\} $.
The carpet-hole toppling procedure of Section~\ref{sub:procedure} chooses the next hot particle from blocks $ 1,\dots,i $ whenever possible.
When it is possible, the choice depends only on the configuration of particles (carpet, hole, frozen, thawed) in these blocks.
When it is not possible, the procedure will either stop, or the hot particles will be chosen in blocks $ i+1,\dots,n $ until the very moment when a successful emission from block~$ i+1 $ to block~$ i $ occurs.
From this point on, the dynamics will again be restricted to blocks $ 1,\dots,i $ until the next time when there are no thawed particles in this region, and so on.

Among the sequence of topplings performed when following this the carpet-hole toppling procedure on $ \dn $, consider only the subsequence that is related to emissions starting in blocks $ 1,\dots,i $.
From the above considerations, we see that this subsequence is only affected by what happens in blocks $ i+1,\dots,n $ by the input of particles arriving at site $ iK+a $.
If in total there are $ M^n_i $ particles arriving from block~$ i+1 $, these particles are the ones with least priority to be chosen as the next hot particle among all other thawed particles.
So, to determine what happens during emissions in blocks~$ 1,\dots,i $, it does not matter whether these $ M^n_i $ particles were being sent back from blocks $ i+1,\dots,n $ one by one and upon request, or whether they had been sitting at $ x=iK+a $ all the time waiting for their turn to become hot.

The only other way that emissions in block~$ i+1 $ could possibly affect the outcome of the procedure on blocks $ 1,\dots,n $ is by using instructions from the stacks $ \xi $ in the transit region.
But this was taken care of by putting two stacks on each site of transit regions.

Since
the procedure on $ \dn $ starting from $ \eta_0^{\rm neat} $
results in the same sequence of topplings as
the procedure on $ D_i $ starting from $ \eta_0^{\rm neat}(M_i^n,i) $
if we ignore the topplings performed during attempted emissions in blocks $ i+1,i+2,\dots,n $, and these sequences determine $ S_i^n(0) $ and $ S_i(M_i^n) $ in the same way, these numbers must coincide.
Similarly for $ L_i^n(0) $ and $ L_i(M_i^n) $.
\end{proof}


The connection between the dynamics of all blocks and the observable $ \sn $ is the following.
\begin{lemma}
\label{lemma:totalfrozen}
Almost surely,
\begin{equation}
\label{eq:totalunavailable}
\sn = \sum_{i=1}^n S_i(M^n_i).
\end{equation}
\end{lemma}
\begin{proof}
This follows straight from the definition of $ \sn $ and Lemma~\ref{lemma:addsame}.
\end{proof}

\subsection{Mass balance equations and proof of the main estimate} \label{chemtrails}

We finally prove Proposition~\ref{prop:fewsleep}.
Rather than working with the complicated structure of $\boldsymbol{M}^n$, we make a bound over all deterministic vectors $\boldsymbol{m}$ that satisfy certain constraints.
A vector $\boldsymbol{m} = (m_0, m_1, \ldots, m_n)$ is said to satisfy the \emph{mass balance equations} if 
\begin{equation}
\label{eq:mbe}
 L_i(m_i) = m_{i-1} \text{ for } i = 1, 2, \ldots, n.
\end{equation}
Note that $\boldsymbol{M}^n$
always satisfies the mass balance equations.
Note also that $ M^n_0 \leq \frac{n}{2} $.
To prove Proposition~\ref{prop:fewsleep}, we rely on the following estimate for a single block.
Note the conditioning on $ \cF_{i-1} $.
\begin{lemma}
\label{lem:singleblock}
Given $ \lambda<\infty $, there exist $ a $ and $ K $ such that the following holds.
For every $ n\in 2\N $ and every $ = 1, 2, \ldots, n$, a.s.,
\begin{equation} \label{nwa}
\sup_{\ell \geq 0} \sum_{m_i = 0}^\infty \E \left[ e^{16S_i(m_i)} 1\{L_i(m_i) = \ell \} \,\middle|\, \cF_{i-1} \right] < e^3.
\end{equation}
\end{lemma}

The proof is given in the next section.

\begin{proof}
[Proof of Proposition~\ref{prop:fewsleep}]
Recall that $\sn$ is the number of frozen free particles remaining in $\dn$ after we apply the the carpet-hole toppling procedure to this region. 
By decomposing over the $ M^n_i $'s and using Lemma~\ref{lemma:totalfrozen} we get
\begin{multline}
\E e^{16\sn}
 \\ =
\E \bigg[\sum_{m_0=0}^{ n/2 } \sum_{m_1} \cdots \sum_{m_{n}} \prod_{i=1}^{n} e^{16S_i(m_i)} 1\{m_i = M^n_i\} 1\{L_i(m_i) = m_{i-1}\} \bigg]
 \\ \leq
\E \bigg[\sum_{m_0=0}^{ n/2 } \sum_{m_1} \cdots \sum_{m_{n}} \prod_{i=1}^{n} e^{16S_i(m_i)} 1\{L_i(m_i) = m_{i-1}\} \bigg]
.
\end{multline}
We will inductively show that
$$
\E \bigg[\sum_{m_0=0}^{ n/2 } \sum_{m_1} \cdots \sum_{m_{\J}} \prod_{i=1}^{\J} e^{16S_i(m_i)} 1\{L_i(m_i) = m_{i-1}\} \bigg]\leq n e^{3\J}.
$$
For $\J=0$ this sum is $1+n/2$ so the base case of our induction is true.
Assume that the inequality is true for $\J-1$.
By conditioning on $\cF_{\J-1}$ and using Lemma~\ref{lem:singleblock}, we get
\begin{eqnarray}
\lefteqn{\E \bigg[ \sum_{m_0} \sum_{m_1} \cdots \sum_{m_{\J}} 
 \prod_{i=1}^{\J} e^{16S_i(m_i)} 1\{L_i(m_i) = m_{i-1}\} \bigg] }\hspace{.5in}&& \nonumber \\
& \leq & \E \bigg[\sum_{m_0} \sum_{m_1} \cdots \sum_{m_{\J-1}} \prod_{i=1}^{\J-1} e^{16S_i(m_i)} 1\{L_i(m_i) = m_{i-1}\} \nonumber\\
&&\hspace{.5in} \cdot \sup_{\ell \geq 0} \sum_{m_\J=0}^{\infty} \E\left[e^{16 S_\J(m_\J)} 1\{L_\J(m_\J) = \ell\} \middle| \cF_{\J-1}\right] \bigg] \nonumber\\ 
& \leq & \E \left[ e^3 \sum_{m_0} \sum_{m_1} \cdots \sum_{m_{\J-1}} \prod_{i=1}^{\J-1} e^{16S_i(m_i)} 1\{L_i(m_i) = m_{i-1}\}\right] \nonumber\\
& \leq & n\cdot e^{3} e^{3(\J-1)} \label{ftp}\nonumber = n \cdot e^{3\J} \nonumber
.
\end{eqnarray}
Using Markov's inequality on our previous calculation with $\J=n$ we obtain 
\begin{align*}
\p(\sn \geq n / 4)
& = \p(e^{16\sn} \geq e^{4n}) \\
& \leq e^{-4n} \E e^{16\sn} \\
& \leq n e^{-4n} e^{3n} 
\leq e^{-cn}
\end{align*}
for large $n$.
\end{proof}

\section{Single block estimate} \label{omari}

In this section we prove Lemma~\ref{lem:singleblock}, the single block estimate.
The key idea is that, however large $ \lambda $ may be, the expected length of a simple random walk excursion is larger.

\subsection{Re-indexing the $L$ and $S$ counters} 

First we fix a block~$i$.
Then we let $j$ be the number of times (counted with multiplicity) that a new particle in block~$i$ has been designated as the the hot particle. We refer to each new designation of the hot particle and its subsequent moves as an \emph{attempted emission}. We re-index the left-emission and frozen particle counters $ L_i(m) $ and $ S_i(m) $ by respectively defining (with a slight abuse of notation) $L(j)$ and $S(j)$ as the total number of particles emitted from block~$ i $ to block~$ i-1 $ and the number of frozen free particles present at block~$ i $, after $ j $ attempted emissions in block~$i$, starting from configuration $ \eta_0^{\rm neat}(\infty, i) $.

Additionally, let $ \w(j) \in [0,a] $ denote the position of the hole after the $j$th attempted emission.
Note that
\begin{equation}
\label{eq:frozenifata}
\w(j)=a
\quad
\Longleftrightarrow
\quad
S(j)=1
.
\end{equation}


Let $ \cG_j $ denote the $ \sigma $-field generated by all the information revealed up to the end of the $ j $th emission attempt in block~$ i $ plus the information encoded by $ \cF_i $. Denote $ \P[\, \cdot\, |\cF_{i-1}] $ by $ \tP $.

We now choose the sizes of the blocks and transit regions in the following way.
Given $\lambda > 0$, set 
\begin{equation} \label{aclu}
a = 12 \lceil e^{100(\lambda +1)}\rceil
\qquad
\text{ and }
\qquad
K = a^2.
\end{equation}

The following lemmas, which estimate deviations for the hole process, will be combined to prove Lemma~\ref{lem:singleblock}. 

\begin{lemma}
\label{walklemma2} For any $j$, the event $\w(j-1) \in [0, a/2] \cup \{a\}$ is $\cG_{j-1}$ measurable 
and, almost surely,
\begin{equation}
\label{eq:usualtransition}
\tP( \w(j)>a/2 \,|\, \cG_{j-1} )1\{\w(j-1) \in [0, a/2] \cup \{a\}\} < e^{-100}.
\end{equation} 
\end{lemma}
Also, 
\begin{lemma}
\label{walklemma1}
For any $j$,
almost surely,
\begin{equation}
\label{eq:unlikelystate}
\tP( a/2 < \w(j) < a \,|\, \cG_{j-1} ) < e^{-100}.
\end{equation}
\end{lemma}

Additionally, the probability of emitting to the left for every two attempted emissions is bounded from below. 

\begin{lemma} \label{walklemma3} For any $j$, the following almost sure bound holds:
\begin{equation}
\tP( L(j+2)>L(j) \,|\, \cG_{j} ) \geq \frac{1}{3}.
\end{equation}
\end{lemma}

Before proving the lemmas, we show how they can be used to prove Lemma~\ref{lem:singleblock}.

\begin{proof}
[Proof of Lemma~\ref{lem:singleblock}]

Recall each particle added at site $ iK+a $ is thawed and, because of the left-most priority policy, the $ m $th particle added to $ iK+a $ in $ \eta_0^{\rm neat}(\infty, i) $ will only become hot after the $ (m-1) $st particle has been designated as the hot particle \emph{and} there are no other thawed free particles to the left of $ iK+a $.
Moreover, it will become hot (and increase $ j $) before the $ (m+1) $st such particle becomes hot.
Thus, each value of $ m $ will correspond to at least one value of $ j $, whereas multiple values of $ j $ can correspond to the same $ m $. Hence,
\[
\sum_{m = 0}^\infty e^{16 S_i(m)} 1\{L_i(m) = \ell\}
\leq
\sum_{j = 0}^\infty e^{16 S(j)} 1\{L(j) = \ell\}
.
\]
So the sum in~\eqref{nwa} can be bounded from above by
\begin{equation}
\label{sbe_reindex}
\leq
\sup_{\ell \geq 0} \sum_{j = 0}^\infty \E \left[ e^{16 S(j)} 1\{L(j) = \ell\} \middle| \cF_{i-1}\right].
\end{equation}
For each $\ell \in \N_0$, let $\tau_\ell$ be the smallest $j$ for which $L(j) = \ell$. Now fix $\ell \in \N$, and re-write the above sum as 
\begin{equation}
\label{sbe_expression}
=
\E \left[\sum_{j = \tau_{\ell}}^{\tau_{\ell +1} - 1}e^{16 S(j)} \ \middle| \ \cF_{i-1}\right].
\end{equation}

Recall that $ \tE $ denotes $ \E[\, \cdot\, |\cF_{i-1}] $. We first rewrite this expectation as
\begin{equation}
=
\tE \sum_{j=0}^{\infty} 1\{ \tau_{\ell+1}-\tau_{\ell} > j \} [1+(e^{16}-1)1\{S(\tau_\ell + j=1)\}]
.
\end{equation}
Define $ G_\ell $ as the event that $ S(j)=0 $ for all $ j $ such that $ L(j)=\ell $, and bound the above expectation by
\begin{equation}
\leq
\tE \sum_{j=0}^{\infty} 1\{ \tau_{\ell+1}-\tau_{\ell} > j \} [ 1\{G_\ell\} + e^{16} 1\{G_\ell^c\} ]
,
\end{equation}
which equals
\begin{multline}
=
\tE [(\tau_{\ell+1}-\tau_{\ell}) 1\{ G_\ell \}]
+
\tE \sum_{j=0}^{\infty} e^{16} 1\{ \tau_{\ell+1}-\tau_{\ell} > j \} 1\{G_\ell^c\}
. \label{largement} 
\end{multline}
The first term in \eqref{largement} is bounded by $ \tE [\tau_{\ell+1}-\tau_{\ell}] $.
By Lemma~\ref{walklemma3},
\begin{equation}
\label{eq:twicegeom}
\p \big[ \tau_{\ell+1}-\tau_\ell > j \,\big|\, \cG_{\tau_{\ell}} \big] \leq (\tfrac{2}{3})^{\lfloor j/2 \rfloor}
\end{equation}
for every $ j \in \N $,
hence
\begin{equation}
\label{eq:six}
\tE [\tau_{\ell+1}-\tau_{\ell}] < 6.
\end{equation}

For the second term in \eqref{largement}, we split the sum at $ k := 600 $ and bound one of the indicators by $ 1 $, obtaining 
\begin{equation}
\label{donc}
\leq
k e^{16} \tP(G_\ell^c) + \sum_{j=k}^{\infty} e^{16} \tP (\tau_{\ell+1}-\tau_{\ell} > j).
\end{equation}
as an upper bound for the last term in \eqref{largement}.


By~\eqref{eq:twicegeom}, the sum in \eqref{donc} is bounded by 
\begin{equation}
\label{eq:tail}
\leq \sum_{j = k}^\infty e^{16} (\tfrac{2}{3})^{\lfloor \tfrac{j}{2} \rfloor} < 1.
\end{equation}
The last inequality follows from our choice of $ k $.

To conclude the proof, we now show that
\begin{equation}
\label{eq:kexpgood}
\tP(G_\ell^c) < \frac{2}{k e^{16}}.
\end{equation}
We consider two events that together imply $ G_\ell $.
First, using~\eqref{eq:twicegeom} and union bound,
\[
\tP( \tau_{\ell+1}-\tau_{\ell-2} > k \,|\, \cG_{\tau_{\ell-2}} )
\leq
3 (\tfrac{2}{3})^{k/6} < \frac{1}{ke^{16}}
.
\]




Second, using Lemmas~\ref{walklemma1} and~\ref{walklemma2},
\begin{multline}
\nonumber
\tP( \w(j) \in [0,\tfrac{a}{2}] \ \forall j=\tau_{\ell-2}+2,\dots,\tau_{\ell-2}+k \,|\, \cG_{\tau_{\ell-2}} )
\geq \\ \geq
(1 - e^{-100})
(1 - e^{-100})^{k-1}
>
1 - ke^{-100}
>
1 - \frac{1}{ke^{16}}
.
\end{multline}
Indeed,
the first $ (1-e^{-100}) $ term is for $ \w(\tau_{\ell-2}+1) \not \in (\frac{a}{2},a) $, and is provided by Lemma~\ref{walklemma1}.
Each of the other $ k-1 $ terms is for $ \w(\tau_{\ell-2}+j) \in [0,\frac{a}{2}] $, and is provided by Lemma~\ref{walklemma2} after conditioning on $ \cG_{\tau_{\ell-2}+j-1} $.

Occurrence of the above event, together with the event $ \tau_{\ell+1} \leq \tau_{\ell-2}+k $, imply that $ \w(j)\in[0,\frac{a}{2}] $ for all $ j $ such that $ L(j)=\ell $.
This in turn implies the event $ G_\ell $, because of~\eqref{eq:frozenifata}.
Hence, combining the two last estimates, we get~\eqref{eq:kexpgood}.

Finally, putting estimates~\eqref{eq:six}, \eqref{donc}, \eqref{eq:tail} and \eqref{eq:kexpgood} together, \eqref{largement} is bounded by
$ 6 + 1 + 2 < e^3 $.
Thus \eqref{sbe_expression} and \eqref{nwa} are also bounded by $e^3$. This proves the lemma.
\end{proof} 

\subsection{A Markov chain} 

Define $\whole_t$ to be a Markov chain on $[0, \infty]$ with initial condition 
$\whole_0 = 0$.
Let
$\{Z_1, Z_2, \ldots\}$ be an iid sequence of random variables with common distribution $Z$ such that
\begin{equation}
\p(Z = z) = \frac{1}{z(z+1)}, \hspace{5pt}z = 1, 2, \ldots 
\end{equation}
This distribution $Z$ is the distribution of the maximum distance away from 0 reached by a simple random walk excursion started from 0.

For each $v$ define the distribution $Y_v$ which takes values
\begin{equation}
\label{eq:jumps}
\begin{cases}
+1, & \text{ with probability } \frac{\lambda}{\lambda+1}, \\
0, & \text{ with probability } \frac{1/2}{\lambda+1}, \\
- \min(Z,v) & \text{ with probability } \frac{1/2}{\lambda+1}.
\end{cases} \end{equation}
These values represent (roughly) the change in the position of the hole if the free particle
falls asleep without moving, moves to the right and returns to the hole, or
moves to the left and returns to the hole.

Let $$\delta= \left(\frac{1}{\lambda+1}\right)\left(\frac{1}{2}\right)\left( \frac{1}{K-2a}\right).$$ This is an upper bound on the probability of an emission to the right (or to the left.) Emissions correspond with either $Y_v=0$ or $Y_v=-v$.
For each $v$ define the distribution $\tilde Y_v$ which takes values
\begin{equation}
\label{eq:jumpstwo}
\begin{cases}
+1, & \text{ with probability } \frac{\lambda}{\lambda+1}+\delta, \\
0, & \text{ with probability } \frac{1/2}{\lambda+1}, \\
- k & \text{ with probability } \frac{1/2}{\lambda+1}\frac{1}{k(k+1)} \text{ for $k =1\dots v-1$}\\
- v & \text{ with probability } \frac{1/2}{\lambda+1}-\delta.
\end{cases} \end{equation}
These values give a distribution which stochastically dominates the change in the position of the hole conditioned on there being no emission.
Then set
the distribution of $\whole_{t+1} - \whole_t$ to be $Y_{\whole_t}.$

For small values $v$ the expected drift 
$$\E(\whole_{t+1} - \whole_t \ | \ \whole_t=v)=\E(Y_v)>0$$
but for values $v\geq a/3$ we get the following lemma.

\begin{lemma} \label{cal} For $v\geq a/3$
$$
\E(Y_v),\E(\tilde Y_v)\leq -40.$$
\end{lemma}

\begin{proof}
Note that 
\begin{align} {\E}[Y_v] & \leq 1- \frac{1}{(\lambda +1)} - \frac{1}{2(\lambda+1)} \sum_{k=1}^{a/3} \frac{1}{k(k+1)} \cdot k \\
& \leq 1- \frac{1}{(\lambda +1)} -\frac{\log(a/3)- \log 2}{2(\lambda +1)} \\
& \leq 1 - \frac{1}{(\lambda +1)}- \frac{100(\lambda +1)- \log 2}{2(\lambda +1)} \\
& \leq 1 - \frac{1}{(\lambda +1)} -49 + \frac{ \log 2}{2(\lambda +1)} \\
& \leq -48.
\qedhere
\end{align}
Using that $v \leq a$ and the definitions of $Y_{a/3}$ and $\tilde Y_{a/3}$ we also get that $${\E}[\tilde Y_v]={\E}[Y_v]+(v+1)\delta \leq -40.$$ \end{proof}

\begin{lemma} \label{sultans} 
Let random variables $\{\tilde Y_{ a/3 }(i)\}_{i=1}^{\infty}$ be i.i.d.\ with distribution $\tilde Y_{ a/3 }.$ 
\begin{equation} \label{defund}
\p \left (\sum_{i=1}^{a/6 } \tilde Y_{ a/3 }(i)\geq -a/6\right ) \leq e^{-a}.
\end{equation}
and 
\begin{equation} \label{spd}
\p \left (\sum_{i=1}^{a/2 } \tilde Y_{ a/3 }(i)>-2a/3\right ) \leq e^{-a}.
\end{equation}
\end{lemma}

\begin{proof}
We will use the following Hoeffding bound.
Suppose $(Y_i)_{i=1}^\infty$ are i.i.d.\ random variables with mean $\nu$ and $Y_i \in [-b, 1]$. Fix $\gamma > 0$ such that $\gamma b$ is an integer, and assume $\nu < - \gamma^{-1}$. Then 

\begin{equation} \p\left(\sum_{i=1}^{\gamma b} Y_i > - b \right) \leq \exp(-2 \gamma (1 + \gamma \nu)^2 b). \end{equation}
Then using $a/6$ instead of $b$, $\gamma=1$ and $\nu=-40$ we get the first inequality.
Using $2a/3$ instead of $b$, $\gamma=3/4$ and $\nu=-40$ we get the second inequality.
\end{proof}

Let $A_t$ be the event that the first $t'>t$ such that $\whole_t=a$ is less than
 the first $t'>t$ such that $\whole_t \leq a/3$.
 \begin{lemma} \label{barricade}
For any value $v \in (a/3,a/2]$ and $t \in \N$. 
$$\p(A_t \ | \ \whole_t=v)\leq e^{-a}.$$
\end{lemma}

\begin{proof}
If $A_t$ occurs then $t'-t\geq a/2$. For each $s \in [t,t')$ we have
$\whole_s>a/3$.
Thus for every $s \in [t,t')$ we have that the distribution of $\whole_{s+1}-\whole_s$ is dominated by $Y_{\lfloor a/3 \rfloor}$ .
So from this starting configuration $A_t$ requires that the sum of $a/2 $ independent random variables with distributions dominated by $Y_{\lfloor a/3 \rfloor}$ is at least $-a/6$.
We bound this by the probability that the sum of $a/2$ independent random variables with distribution of $\tilde Y_{\lfloor a/3 \rfloor}$ is at least $-a/6$.
By Lemmas~\ref{cal} and~\ref{sultans} 
we obtain the probability
$$ \text{} \hspace{1.1in} 
\p \left (\sum_{i=1}^{a/6 } \tilde Y_{ a/3 }(i)>-a/6\right ) \leq e^{-a}. \hspace{1.1in} \qedhere$$
\end{proof}

\subsection{Analysis of the carpet-hole toppling procedure} 
We say a step of the carpet-hole toppling procedure is the hot particle starting at the hole and either 
\begin{enumerate}
\item falling asleep,
\item emitting, or
\item moving away from the hole and returning to it.
\end{enumerate}

Let $T_j$ be the number of steps taken by carpet-hole procedure between the $j-1$st failure/emission and the $j$th failure/emission.

\begin{lemma} \mbox{} \label{pride}
For all $v$ we have
 \begin{equation}\label{dont} \p(T_j > a^3 \ | \ \w(j-1)=v)<\frac{1}{a}.\end{equation} If $v\leq a/2$ then 
 \begin{equation}\p(T_j < a/2 \ | \ \w(j-1)=v)<\frac{1}{4a}. \label{mounds} \end{equation}
\end{lemma}

\begin{proof}
At each step the chance of an emission is at least $\frac{1}{(\lambda+1)(K+a)}>a^{-2.1}.$ Then the first inequality follows from 
$$\left(1-\frac{1}{(\lambda+1)(K+a)}\right)^{a^3}<(1-a^{-2.1})^{a^3}<2e^{ -a^{.9}}<\frac{1}{a}.$$

For the second inequality as $v<a/2$ there can be no failure in the first $a/2$ steps. Thus $T_j$ can only be $< a/2$ if there is an emission. At each step the chance of an emission is at most $\frac{1}{(\lambda+1)(K-a)}>\frac{1}{2(\lambda+1)K}.$
By a union bound the probability of an emission in the first $a/2$ steps is at most
\[
(a/2)\frac{1}{2(\lambda+1)K} \leq \frac{a}{4K}= \frac1{4a}. \qedhere
\]
\end{proof}

\begin{proof}[Proof of Lemma~\ref{walklemma2}]
\
Recall that $ \cG_j $ contains all the information revealed up to the end of the $ j $th emission attempt in block~$ i $ plus the information encoded by $ \cF_i $. 
Thus for each $j$, $\w(j)$ is ${\mathcal G}_{j}$ measurable.
The process $\w(j) $ is a Markov process that is related to, but different from the one described in the previous section. That process continues until there is an emission or a failure. (Emissions only occur when $\whole(j)-\whole(j-1) \in \{0, w_{j-1}\}$ so they are correlated with the $\whole(j)$ process.) 
As $\cG_{j-1}$ contains no information after the $j-1$st emission, conditioning on ${\mathcal G}_{j-1}$ is equivalent to conditioning on $\w(j-1)$.


Thus it is sufficient to show that for any value $v \in [0, a/2] \cup \{a\}$,
$$\p(\w(j) >a/2 \ | \ \w(j-1)=v) < \frac{4}{a}<e^{-100}.$$

{\bf Case $ v=a $.} If $v=a$ then there is a frozen free particle in the hole, every location in the block and the adjacent transit regions has a particle, and the hot particle is performing simple random walk until it leaves one of the two transit regions at $K-a$ or $-K+a$.
Since the hot particle starts at $\pm a$, the probability that it visits the entire block before being emitted is at least $$1 - \frac{2a}{K-2a}>1-\frac{4a}{K}=1-\frac{4}{a}.$$ On this event, the hole is reset to 0, so $\w(j) = 0$. 
This establishes the first case.

{\bf Case $ v \leq a/3 $.} We will prove the following bounds. 
\begin{equation} \label{sdot}\tP( \w(j)>a/2 \cap \{T_j \in [0,a^3)\} \,|\, \w(j-1)=v ) < \frac{1}{a} \end{equation} and
\begin{equation}\tP( \w(j)>a/2 \cap \{T_j \geq a^3\} \,|\, \w(j-1)=v ) < \frac{1}{a}. \label{removal} \end{equation}
Then this case will follow by a union bound.

Consider the process with the hole started at $a/3$. We calculate the probability that the hole reaches $a/2$ before it becomes less than $a/3$. This takes at least $a/6$ steps. At each of these steps the hole is at or to the right of $a/3$. Thus the distribution of the movement of the hole, conditioned on no emission, is stochastically bounded by $\tilde Y_{a/3}$ from \eqref{eq:jumpstwo}. The probability that the hole reaches $a/2$ before it reaches $a/3-1$ is therefore bounded by the probability that the sum of $a/6$ copies of $\tilde Y_{a/3}$ is at least 0. By Lemma~\ref{sultans} this probability is bounded by $e^{-a}$. 

To prove \eqref{sdot} we note that if $\w(j)>a/2 \cap \{T_j \in [0,a^3)\}$ occurs then 
one of the first $a^3$ times the hole is at $a/3$ it moves to $a/2$ before returning to the left of $ a/3$. By the union bound and the fact that $a>1$ this has probability at most $a^3 e^{-a}<1/a.$ 
The bound in \eqref{removal} follows straight from the first half of Lemma~\ref{pride}.


{\bf Case $ v \in (a/3,a/2] $.} We proceed as in the previous case except that we break the event in \eqref{sdot} into two pieces; one with $T_j<a/2$ and one with 
$T_j \in [a/2,a^3)$. Not that a failure cannot occur with $T_j < a/2$. The probability that an emission occurs in the first $a/2 -1$ steps is bounded by $1/4a$ by the second half of Lemma~\ref{pride}. 

Next we bound the probability that an emission does not occur in the first $a/2 -1$ steps and the hole is always to the right of $a/3$ in the first $a/2-1$ steps. In this case the step distribution of the hole is bounded by $\tilde Y_{a/3}$. If the hole is never $\leq a/3$ then the sum of the first $a/6$ steps (conditioned on no emission) is at least $-a/6$ with probability at most $e^{-a}$ by \eqref{defund} in Lemma~\ref{sultans}.
Once the hole is at (or to the left of) $a/3$ we proceed as in the previous case and we get that 
 $$\p(\w(j)>a/2 \cap \{T_j<a/2\})<1/a$$ and
 $$\p(\w(j)>a/2 \cap \{T_j\in[a/2,a^3]\})<2/a.$$
From the first half of Lemma~\ref{pride} we get that
 $$\p(\w(j)>a/2 \cap \{T_j>a^3\})<1/a.$$
 Then the union bound completes the proof.
\end{proof}

Next we prove Lemma~\ref{walklemma1}.

\begin{proof}
[Proof of Lemma~\ref{walklemma1}] \text{ }
As mentioned in the previous proof conditioning on ${\mathcal G}_{j-1}$ is equivalent to conditioning on $\w(j-1)$.


The case $j = 0$ is trivial, since the hole starts at position $0$. So suppose $j \geq 1$. First, if $\w(j-1) \in [0, a/2] \cup \{a\}$, then Lemma~\ref{walklemma2} immediately gives a bound for the probability that $\w(j) > a/2$. 

So assume $ \w(j-1) \in (a/2, a)$. The proof is very similar to
 the case $v> a/3$ in the proof of Lemma~\ref{walklemma2}.
If there is a failure before an emission then 
$\w(j)=a$. So we only need to consider the case that there is an emission first. The probability that there is an emission in the first $a/2$ steps is bounded by the second half of Lemma~\ref{pride}.

The probability that there is no emission in the first $a/2$ steps and
and the hole is never in $[0,a/3]$
in that time is at most $e^{-a}$ 
by the second part of Lemma~\ref{sultans}.
 The bounds are the same as in the proof of Lemma~\ref{walklemma2}.
%
\end{proof}

Finally, we complete our argument by proving Lemma~\ref{walklemma3}. 

\begin{proof}[Proof of Lemma~\ref{walklemma3}] \text{ }
Again conditioning on ${\mathcal G}_{j}$ is equivalent to conditioning on $\w(j)$.
Each time a failure occurs, by~\ref{prop:atmostone} the next hot particle to be activated in block~$i$ automatically is emitted successfully.
Thus at least one successful emission occurs for every two attempted emissions. Each particle starts from inside the block, which is an interval of length $2a$, while the neighboring blocks are distance $K$ away on either side. Thus, the classical gambler's ruin probabilities for simple random walk shows that each successful emission has probability at least $\frac{1}{2}-\frac{a}{2K} \geq \frac{1}{3}$ of being to the left. The result follows.
\end{proof}

\section*{Acknowledgments}
The authors would like to thank the International Centre for Theoretical Sciences, Bangalore. This work was initiated in 2019 at the program on Universality in Random Structures: Interfaces, Matrices, Sandpiles, held ICTS (Code: ICTS/urs2019/01). 
We would like to thank Riddhi Basu, Amine Asselah and Bruno Schapria for inspiring discussions.
C.H.\ was supported by NSF grant DMS-1712701.
J.R.\ was  was supported by NSF grants DMS-1444084 and DMS-1712701.

\bibliographystyle{bib/leoabbrv}
\bibliography{bib/leo,bib/leo2}

\begin{thebibliography}{10}
\expandafter\ifx\csname urlstyle\endcsname\relax
  \providecommand{\doi}[1]{doi:\discretionary{}{}{}#1}\else
  \providecommand{\doi}{doi:\discretionary{}{}{}\begingroup
  \urlstyle{rm}\Url}\fi

\bibitem{AmirGurel-Gurevich10}
G.~\textsc{Amir}, O.~\textsc{Gurel-Gurevich}.
\newblock \emph{On fixation of activated random walks}.
\newblock Electron Commun Probab \textbf{15}:119--123, 2010.
\newblock \doi{10.1214/ECP.v15-1536}.

\bibitem{AsselahSchapiraRolla19}
A.~\textsc{Asselah}, B.~\textsc{Schapira}, L.~T. \textsc{Rolla}.
\newblock \emph{Diffusive bounds for the critical density of activated random
  walks}, 2019.
\newblock Preprint. \href{http://arxiv.org/abs/1907.12694}{arXiv:1907.12694}.

\bibitem{BasuGangulyHoffman18}
R.~\textsc{Basu}, S.~\textsc{Ganguly}, C.~\textsc{Hoffman}.
\newblock \emph{Non-fixation for conservative stochastic dynamics on the line}.
\newblock Comm Math Phys \textbf{358}:1151--1185, 2018.
\newblock \doi{10.1007/s00220-017-3059-7}.

\bibitem{BasuGangulyHoffmanRichey19}
R.~\textsc{Basu}, S.~\textsc{Ganguly}, C.~\textsc{Hoffman}, J.~\textsc{Richey}.
\newblock \emph{Activated random walk on a cycle}.
\newblock Ann Inst Henri Poincar\'{e} Probab Stat \textbf{55}:1258--1277, 2019.
\newblock \doi{10.1214/18-aihp918}.

\bibitem{BondLevine16}
B.~\textsc{Bond}, L.~\textsc{Levine}.
\newblock \emph{Abelian networks i. foundations and examples}.
\newblock SIAM J Discrete Math \textbf{30}:856--874, 2016.
\newblock \doi{10.1137/15M1030984}.

\bibitem{CabezasRolla20}
M.~\textsc{Cabezas}, L.~T. \textsc{Rolla}.
\newblock \emph{Avalanches in critical activated random walks}, 2020.
\newblock Preprint. \href{http://arxiv.org/abs/2008.05783}{arXiv:2008.05783}.

\bibitem{CabezasRollaSidoravicius14}
M.~\textsc{Cabezas}, L.~T. \textsc{Rolla}, V.~\textsc{Sidoravicius}.
\newblock \emph{Non-equilibrium phase transitions: Activated random walks at
  criticality}.
\newblock J Stat Phys \textbf{155}:1112--1125, 2014.
\newblock \doi{10.1007/s10955-013-0909-3}.

\bibitem{CabezasRollaSidoravicius18}
---{}---{}---.
\newblock \emph{Recurrence and density decay for diffusion-limited annihilating
  systems}.
\newblock Probab Theory Relat Fields \textbf{170}:587--615, 2018.
\newblock \doi{10.1007/s00440-017-0763-3}.

\bibitem{CandelleroGangulyHoffmanLevine17}
E.~\textsc{Candellero}, S.~\textsc{Ganguly}, C.~\textsc{Hoffman},
  L.~\textsc{Levine}.
\newblock \emph{Oil and water: a two-type internal aggregation model}.
\newblock Ann Probab \textbf{45}:4019--4070, 2017.
\newblock \doi{10.1214/16-AOP1157}.

\bibitem{CandelleroStaufferTaggi20}
E.~\textsc{Candellero}, A.~\textsc{Stauffer}, L.~\textsc{Taggi}.
\newblock \emph{Abelian oil and water dynamics does not have an absorbing-state
  phase transition}.
\newblock Trans Amer Math Soc \textbf{to appear}, 2020.
\newblock \href{http://arxiv.org/abs/1901.08425}{arXiv:1901.08425}.

\bibitem{DickmanRollaSidoravicius10}
R.~\textsc{Dickman}, L.~T. \textsc{Rolla}, V.~\textsc{Sidoravicius}.
\newblock \emph{Activated random walkers: Facts, conjectures and challenges}.
\newblock J Stat Phys \textbf{138}:126--142, 2010.
\newblock \doi{10.1007/s10955-009-9918-7}.

\bibitem{Jarai18}
A.~A. \textsc{J\'{a}rai}.
\newblock \emph{Sandpile models}.
\newblock Probab Surv \textbf{15}:243--306, 2018.
\newblock \doi{10.1214/14-PS228}.

\bibitem{Rolla19}
L.~T. \textsc{Rolla}.
\newblock \emph{Activated random walks on {$Z^d$}}, 2019.
\newblock Preprint. \href{http://arxiv.org/abs/1906.05037}{arXiv:1906.05037}.

\bibitem{RollaSidoravicius12}
L.~T. \textsc{Rolla}, V.~\textsc{Sidoravicius}.
\newblock \emph{Absorbing-state phase transition for driven-dissipative
  stochastic dynamics on {$Z$}}.
\newblock Invent Math \textbf{188}:127--150, 2012.
\newblock \doi{10.1007/s00222-011-0344-5}.

\bibitem{RollaSidoraviciusZindy19}
L.~T. \textsc{Rolla}, V.~\textsc{Sidoravicius}, O.~\textsc{Zindy}.
\newblock \emph{Universality and sharpness in activated random walks}.
\newblock Ann Henri Poincar{\'e} \textbf{20}:1823--1835, 2019.
\newblock \doi{10.1007/s00023-019-00797-0}.

\bibitem{RollaTournier18}
L.~T. \textsc{Rolla}, L.~\textsc{Tournier}.
\newblock \emph{Non-fixation for biased activated random walks}.
\newblock Ann Inst H Poincar{\'e} Probab Statist \textbf{54}:938--951, 2018.
\newblock \doi{10.1214/17-AIHP827}.

\bibitem{Shellef10}
E.~\textsc{Shellef}.
\newblock \emph{Nonfixation for activated random walks}.
\newblock ALEA Lat Am J Probab Math Stat \textbf{7}:137--149, 2010.
\newblock \href{http://alea.impa.br/articles/v7/07-07.pdf}{pdf}.

\bibitem{SidoraviciusTeixeira17}
V.~\textsc{Sidoravicius}, A.~\textsc{Teixeira}.
\newblock \emph{Absorbing-state transition for stochastic sandpiles and
  activated random walks}.
\newblock Electron J Probab \textbf{22}:33, 2017.
\newblock \doi{10.1214/17-EJP50}.

\bibitem{StaufferTaggi18}
A.~\textsc{Stauffer}, L.~\textsc{Taggi}.
\newblock \emph{Critical density of activated random walks on transitive
  graphs}.
\newblock Ann Probab \textbf{46}:2190--2220, 2018.
\newblock \doi{10.1214/17-AOP1224}.

\bibitem{Taggi16}
L.~\textsc{Taggi}.
\newblock \emph{Absorbing-state phase transition in biased activated random
  walk}.
\newblock Electron J Probab \textbf{21}:13, 2016.
\newblock \doi{10.1214/16-EJP4275}.

\bibitem{Taggi19}
---{}---{}---.
\newblock \emph{Active phase for activated random walks on {$\Bbb{Z}^d$},
  {$d\geq3$}, with density less than one and arbitrary sleeping rate}.
\newblock Ann Inst Henri Poincar\'{e} Probab Stat \textbf{55}:1751--1764, 2019.
\newblock \doi{10.1214/18-aihp933}.

\bibitem{Taggi20}
---{}---{}---.
\newblock \emph{Essential enhancements in {A}belian networks: continuity and
  uniform strict monotonicity}, 2020.
\newblock Preprint. \href{http://arxiv.org/abs/2003.00932}{arXiv:2003.00932}.

\end{thebibliography}

\end{document}